\documentclass[12pt]{amsart}
\usepackage{amssymb,amsmath,stmaryrd}
\usepackage{bbm} 
\usepackage{hyperref}
\hypersetup{pdfpagemode=FullScreen,colorlinks=true}
\usepackage{todonotes}
\theoremstyle{plain}
\newtheorem{proposition}{Proposition}[section]
\newtheorem{theorem}[proposition]{Theorem}
\newtheorem{lemma}[proposition]{Lemma}
\newtheorem{corollary}[proposition]{Corollary}

\theoremstyle{definition}
\newtheorem{definition}[proposition]{Definition}
\newtheorem{remark}[proposition]{Remark}
\newtheorem{fact}[proposition]{Fact}
\newtheorem{notation}[proposition]{Notation}
\theoremstyle{remark}

\def\R{\mathbb{R}}

\def\eps{\epsilon}

\def\Xint#1{\mathchoice
      {\XXint\displaystyle\textstyle{#1}}%
      {\XXint\textstyle\scriptstyle{#1}}%
      {\XXint\scriptstyle\scriptscriptstyle{#1}}%
      {\XXint\scriptscriptstyle\scriptscriptstyle{#1}}%
      \!\int}
   \def\XXint#1#2#3{{\setbox0=\hbox{$#1{#2#3}{\int}$}
        \vcenter{\hbox{$#2#3$}}\kern-.5\wd0}}
   \def\fint{\Xint-}

\newenvironment{pf}{\begin{trivlist}\item[]{\bf Proof\ }}
{\mbox{}\hfill\rule{.08in}{.08in}\end{trivlist}}

\begin{document}

\title[Higher degree forms and QC]{Higher degree obstructions to quasiconformal equivalence}

\author[Pansu]{Pierre Pansu}
\address[Pansu]{Laboratoire de Math\'ematiques d'Orsay, CNRS, Universit\'e Paris-Saclay, 91405 Orsay, France}
\email{pierre.pansu@universite-paris-saclay.fr}

\author[Tripaldi]{Francesca Tripaldi}
\address[Tripaldi]{Department of Pure Mathematics, University of Leeds, Woodhouse, LS2 9JT Leeds, UK}
\email{f.tripaldi@leeds.ac.uk}

\maketitle

\hfill \emph{Dedicated to Mario Bonk and Bruce Kleiner on their 63rd birthday}

\section{Introduction}

\subsection{Motivation}

The goal of the present paper is to illustrate the use of higher degree tools in geometric function theory. By higher degree, we mean the use of differential forms of degree $>1$. They first show up in \cite{DS}, which has inspired \cite{IM}. A notable breakthrough exploiting differential forms in all degrees is contained in \cite{BH} and \cite{P}. 

The usefulness of higher degree invariants has been overshadowed by the efficiency of degree $1$ invariants, among which moduli of curve families and capacities. This efficiency is especially visible in the recent paper \cite{FLDNGOP}, which focusses on the issue whether, for a given class of metric spaces, every quasiconformal (QC) mapping is automatically a quasiisometry (QI). There, the notion of \emph{strict conformal parabolicity} is introduced, refining classical conformal parabolicity. A metric measure space is called \emph{liminal conformally parabolic} if it is conformally parabolic but not strictly conformally parabolic (precise definitions will be given in Section \ref{ferrand}). Thus the classical conformal dichotomy hyperbolic/parabolic becomes a trichotomy hyperbolic/liminal parabolic/strictly parabolic. For geodesic Lie groups, in the conformally hyperbolic and strictly parabolic cases, degree $1$ invariants provide an answer. In the conformally liminal parabolic case, they do not. 

More flexibility is expected in the conformally liminal parabolic case than in the other (strictly parabolic and hyperbolic) cases. For instance, Euclidean space is conformally liminal parabolic, and it admits a wealth of non-QI QC self-maps (so do Heisenberg groups). As an other illustration of flexibility, there are spaces (metrics of revolution on $\R^n$) which are conformal to Euclidean $n$-space, and nevertheless have a volume growth exponent different from $n$, hence are not QI to Euclidean space.

However, in the present paper, we show that some rigidity can occur even in the conformally liminal parabolic world, and in presence of large symmetry groups: we exhibit spaces admitting cocompact nilpotent isometry groups, which are conformally liminal parabolic, and nevertheless can be shown not to be QC. For this purpose, we use a higher degree invariant, the conformal cohomology in degree $2$, introduced by V. Gol'dshtein and M. Troyanov in \cite{GoldTroyConf}. 

\subsection{Results}

\begin{theorem}
\label{riemannian}
There exist pairs of Riemannian manifolds $(M,N)$ such that
\begin{enumerate}
  \item Both $M$ and $N$ admit cocompact isometry groups.
  \item $M$ and $N$ are diffeomorphic.
  \item $M$ and $N$ are conformally liminal parabolic. 
  \item $M$ and $N$ differ in $2$-dimensional conformal cohomology $L^{n,n/2}H^2$.
\end{enumerate}
Therefore $M$ and $N$ are not QC.
\end{theorem}

A version of conformal cohomology adapted to contact subRiemannian manifolds, whose QC invariance follows from \cite{KMX}, allows the following contact subRiemannian variant. 

\begin{theorem}
\label{main}
There exist pairs of contact subRiemannian manifolds $(M,N)$ such that
\begin{enumerate}
  \item Both $M$ and $N$ admit cocompact isometry groups.
  \item $M$ and $N$ are smoothly contactomorphic.
  \item $M$ and $N$ are conformally liminal parabolic. 
  \item $M$ and $N$ differ in $2$-dimensional Rumin conformal cohomology $L^{2n+2,n+1}H^2$.
\end{enumerate}
Therefore $M$ and $N$ are not QC.
\end{theorem}

\subsection{The construction}
\label{construction}

The proof relies on facts about the $L^{q,p}$ cohomology of Carnot groups, carefully tailored examples of such groups and various QC and QI invariance properties.

\begin{fact}[\cite{PR}]
\label{fact}
Let $G$ be a Carnot group with homogeneous dimension $Q$. 
\begin{enumerate}
  \item $L^{q,p}H^1(G)=0\iff \frac{1}{p}-\frac{1}{q}\ge\frac{1}{Q}$.
  \item Let $[w_{min},w_{max}]$ denote the range of weights occurring in the Lie algebra cohomology $H^{2}(\mathfrak{g})$. Then
  \begin{enumerate}
  \item $\frac{1}{p}-\frac{1}{q}\ge\frac{w_{max}-1}{Q}\implies L^{q,p}H^2(G)=0$.
  \item $\frac{1}{p}-\frac{1}{q}<\frac{w_{min}-1}{Q}\implies L^{q,p}H^2(G)\not=0$.
\end{enumerate}
\end{enumerate}
\end{fact}

The basic bricks are pairs of Carnot groups $F$ and $G$ whose Lie algebras $\mathfrak{f}$ and $\mathfrak{g}$ satisfy:
\begin{enumerate}
  \item Vectorspace dimension $\mathrm{dim}(\mathfrak{f})=\mathrm{dim}(\mathfrak{g})=m$.
  \item Homogeneous dimension $Q(\mathfrak{f})=Q(\mathfrak{g})=Q>m+1$.
  \item $w_{min}(\mathfrak{f})=3$, $w_{max}(\mathfrak{g})=2$.
\end{enumerate}

\subsection{Riemannian case}

To get the Riemannian examples, one forms products with spheres $M=F\times S^{n-m}$, $N=G\times S^{n-m}$. Then $M$ and $N$ are diffeomorphic $n$-dimensional manifolds. Let us pick $F$- (resp $G$-) invariant Riemannian metrics on $M$ and $N$. 

Since $F$ and $G$ are Carnot groups of homogeneous dimension $Q$, they are liminal $Q$-parabolic. 

Setting $n=Q$, $p=\frac{n}{2}$ and $q=n$, $\frac{1}{p}-\frac{1}{q}=\frac{1}{n}$. If $w_{max}(\mathfrak{g})=2$ and $w_{min}(\mathfrak{f})\ge 3$, then
$$
\frac{w_{max}(\mathfrak{g})-1}{Q}=\frac{1}{n}=\frac{1}{p}-\frac{1}{q}<\frac{2}{n}=\frac{w_{min}(\mathfrak{f})-1}{Q}.
$$
Fact \ref{fact} implies that
$$
L^{n,n/2}H^2(F)\not=0,\quad L^{n,n/2}H^2(G)=0.
$$
Since $n-m>1$, $M$ and $N$ are simply connected. By quasiisometry invariance of $L^{q,p}$ cohomology in degree $2$, \cite{Ducret}, the cohomological properties pass from $F,G$ to $M,N$, hence
$$
L^{n,n/2}H^2(M)\not=0,\quad L^{n,n/2}H^2(N)=0,
$$
We shall show that liminal $Q$-parabolicity also passes. Since $Q$ coincides with their Hausdorff dimension, $M$ and $N$ are conformally liminal parabolic. 

\subsection{Contact subRiemannian case}

To get the contact subRiemannian examples, one forms products with spheres $V=F\times S^{n+1-m}$, $W=G\times S^{n+1-m}$, and one takes projectivized cotangent bundles $M=P(T^*V)$, $N=P(T^*W)$, in order to get contact manifolds with cocompact $F$- (resp. $G$-) actions. Since $V$ and $W$ are diffeomorphic, $M$ and $N$ are contactomorphic.

Setting $2n+2=Q$, $p=n+1$ and $q=2n+2$, $\frac{1}{p}-\frac{1}{q}=\frac{1}{2n+2}$. If $w_{max}(\mathfrak{g})=2$ and $w_{min}(\mathfrak{f})\ge 3$, then
$$
\frac{w_{max}(\mathfrak{g})-1}{Q}=\frac{1}{2n+2}=\frac{1}{p}-\frac{1}{q}<\frac{1}{n+1}=\frac{w_{min}(\mathfrak{f})-1}{Q}.
$$
Fact \ref{fact} implies that
$$
L^{2n+2,n+1}H^2(F)\not=0,\quad L^{2n+2,n+1}H^2(G)=0,
$$
$M$ and $N$ are simply connected. By quasiisometry invariance of $L^{q,p}$ Rumin cohomology in degree $2$ (see \cite{Pansu_cup}, relying crucially on \cite{BFP}), the cohomological properties pass again from $F,G$ to $M,N$, hence
$$
L^{2n+2,n+1}H^2(M)\not=0,\quad L^{2n+2,n+1}H^2(N)=0.
$$
Liminal $Q$-parabolicity will be shown to be QI invariant also for subRiemannian manifolds. Since $Q$ coincides with their Hausdorff dimension, $M$ and $N$ are conformally liminal parabolic.

\subsection{Organization of the paper}

The construction relies on three technical inputs of different natures:
\begin{itemize}
  \item The QI invariance of liminal $p$-parabolicity is established\footnote{For this, a slight twist in the original definition of liminal $p$-parabolicity is required}, see Section \ref{ferrand}.
  \item The conformal Rumin complex for contact subRiemannian manifolds is defined and its QC invariance is proven in Section \ref{rumin}.
  \item Examples of pairs of 3-step Carnot Lie algebras with the required homological and numerical properties are provided in Section \ref{lie}.
\end{itemize}
These three ingredients complete the proofs of Theorems \ref{riemannian} and \ref{main}.

\section{Degree $1$ invariants}
\label{ferrand}

\subsection{QC maps}

\begin{definition} \label{QC metric definition}
A map $f:X\to Y$ between metric spaces is \emph{quasiconformal} (abbreviated QC) if it is a homeomorphism and if there exists $K\ge 1$ such that, for every $x\in X$,
\begin{align*}
\limsup_{r\to 0} \frac{\sup\{ d(f(x),f(x')) : x'\in B(x,r) \} }{ \inf\{ d(f(x),f(x')) : x'\notin B(x,r) \}  } \le K.
\end{align*}
\end{definition}

\subsection{Capacities}

The setting of this section, encompassing both Riemannian, subRiemannian manifolds and many more spaces, is that of metric measure spaces. Extra assumptions will show up later in this section.

In metric measure spaces, one can define $p$-capacities for $p\ge 1$. For a pair of disjoint subsets $E,F\subset X$,
\begin{align*}
&\mathrm{Cap}_p(E,F)\\
&:=\inf\{\int_{X}g^p\,;\,u\in L^1_{loc} \text{ with upper gradient }g,\,u(E)\ge 1,\,u(F)\le 0\}.
\end{align*} 
For a single subset $E\subset X$,
\begin{align*}
\mathrm{Cap}_p(E):=\inf\{\mathrm{Cap}_p(E,F)\,;\, X\setminus F \text{ is bounded}\}.
\end{align*} 

If $E$ or $F$ is discrete, $p$-capacity tends to vanish for $p\le Q$, the Hausdorff dimension of $X$. This is why $E$ and $F$ are often required to be connected and have large, or even infinite, diameter. Here is a slight strengthening of unboundedness and connectedness.

\begin{definition} \label{escaping}
Let $X$ be a metric space. A subset $E\subset X$ is \emph{escaping} if for all bounded sets $B\subset E$, $E\setminus B$ contains an unbounded connected subset.
\end{definition}

\subsection{Liminal $p$-parabolicity}

\begin{definition}[$p$-parabolicity/hyperbolicity, \cite{Troyanov}]
A metric measure space $X$ is \emph{$p$-parabolic} if all bounded sets have vanishing $p$-capacity. Spaces which are not $p$-parabolic are called \emph{$p$-hyperbolic}. 
\end{definition}

\begin{definition}[Strict/liminal $p$-parabolicity]
A metric measure space $X$ is \emph{strictly $p$-parabolic} if 
\begin{enumerate}
  \item $X$ is $p$-parabolic
  \item There exists a pair of escaping subsets $E,F\subset X$ such that
$$
\mathrm{Cap}_p(E,F)<+\infty.
$$
\end{enumerate}
Spaces which are $p$-parabolic but not strictly $p$-parabolic are called \emph{liminal $p$-parabolic}.
\end{definition}

\subsection{Conformal liminal parabolicity}

\begin{definition}[Conformal parabolicity]
A metric measure space $X$ of Hausdorff dimension $Q$ is \emph{conformally hyperbolic / conformally parabolic / strictly conformally parabolic / liminal conformally parabolic} if it is $Q$-hyperbolic / $Q$-parabolic / strictly $Q$-parabolic / liminal $Q$-parabolic.
\end{definition}

We rely on a result of \cite{HK+} which guarantees that QC mappings preserve capacities up to multiplicative constants. Here is the assumption in their result.

\begin{definition}[ul$Q$bg, compare Definition 2.1 in \cite{FLDNGOP}]\label{ulQbg}
A metric measure space $(X,d,\mu)$ has \emph{uniformly locally $Q$-bounded geometry} with data $(Q,R,C_{\rm A}, C_{\rm P},\sigma)$, where $Q\ge1$, $R>0$, $C_{\rm A}\ge1$, $C_{\rm P}\geq 0$, and $\sigma\ge1$, if $X$ is separable, pathwise connected, and locally compact, and if the following two conditions hold:
\begin{enumerate}
  \item \emph{Ahlfors regularity}. $\forall x\in X$, $\forall r<R$,
\begin{align*}
C_{\rm A}^{-1} r^Q \le \mu(B(x,r)) \le C_{\rm A} r^Q ,
\end{align*}
\item \emph{$(1,Q)$-Poincar\'e inequality with sharp dependence on radius}. For every $x\in X$ and for all $r\in(0,R/\sigma)$, if $u:B(x,\sigma r)\to\R$ is a locally integrable function with upper gradient $g$, then
\begin{align*}
\left(\fint_{B(x,r)} |u-u_B| \right) \le C_{\rm P} r \left( \fint_{B(x,\sigma r)} g^Q \right)^{1/Q} ,
\end{align*}
where $u_B := \fint_{B(x,r)} u$ denotes the average of $u$ over $B(x,r)$.
\end{enumerate}
\end{definition}

\begin{proposition}[Theorem 9.10 in \cite{HK+}] \label{QCcap}
On ul$Q$bg spaces, quasiconformal mappings preserve capacities up to a bounded multiplicative factor.
\end{proposition}

Therefore, on ul$Q$bg spaces, conformal hyperbolicity, parabolicity, strict parabolicity and liminal parabolicity are QC invariant, since they depend only on capacities.

\subsection{Coarsening capacities}

Our next goal is to prove the QI invariance of $p$-parabolicity and its variants. This will hold under weaker geometric assumptions.

\begin{definition} \label{pPI}
A metric measure space $(X,d,\mu)$ is \emph{$p$-PI} with data $p\ge 1$, $R>0$, $\sigma\ge 1$, and profiles $m, M, C_{p\mathrm{P}}:(0,+\infty)\to(0,+\infty)$ if
\begin{enumerate}
  \item \emph{Uniform local bounds on measures of balls}. For all $r\in(0,R)$, for all $x\in X$,
  \begin{align*}
m(r)\le \mu(B(x,r)) \le M(r).
\end{align*}
\item \emph{$(p,p)$-Poincar\'e inequality}. For all $x\in X$, for all $r\in(0,R/\sigma)$, if $u:B(x,\sigma r)\to\R$ is a locally integrable function with upper gradient $g$, then
\begin{align*}
\int_{B(x,r)} |u-u_B|^p  \le C_{p\mathrm{P}}(r)\int_{B(x,\sigma r)} g^p.
\end{align*}
\end{enumerate}

\end{definition}

Following \cite{Pansu_Torino}, one defines a coarse version of $p$-energy. A \emph{kernel} on $X$ is a nonnegative bounded Lipschitz function $\psi$ on $X\times X$, which is bounded away from $0$ (resp. has support in) a bounded neighborhood of the diagonal, and satisfies
\begin{align*}
\forall x,y\in X,\quad \int_{X}\psi(z,y)\,dz =\int_{X}\psi(x,z)\,dz=1.
\end{align*}
Given a cocycle (i.e. a function $\alpha$ on $X\times X$ such that $\forall x,y,z\in X$, $\alpha(x,y)+\alpha(y,z)=\alpha(x,z)$), one defines
\begin{align*}
N^{\psi}(\alpha)=\left(\int_{X}|\alpha(x,y)|^{p}\psi(x,y)\,dx\,dy \right)^{1/p}.
\end{align*}
For $p$-PI spaces, one easily checks that, as $\psi$ varies, all these norms on the space $Z(X)$ of cocycles  are equivalent.

\begin{definition}
Coarse $p$-capacities can be defined as follows. Given a kernel $\psi$,
\begin{align*}
\mathrm{Cap}_\psi(E,F)&:=\inf\{N^\psi(\alpha)^p\,;\,\alpha_{|E\times E}=0,\,\alpha_{|F\times F}=0,\,\alpha_{|E\times F}=1\},\\
\mathrm{Cap}_\psi(E)&:=\inf\{\mathrm{Cap}_\psi(E,F)\,;\, X\setminus F \text{ is bounded}\}.
\end{align*}
\end{definition}

\subsection{QI invariance of coarse $p$-parabolicity}

\begin{definition}
Say $X$ is \emph{coarsely $p$-parabolic} if, for some kernel $\psi$, $\mathrm{Cap}_\psi(E)=0$ for every bounded set $E\subset X$.
\end{definition}

Let $f:X\to Y$ be a quasiisometry between $p$-PI spaces. One easily checks (Proposition 1.7 in \cite{Pansu_Torino}) that the map
\begin{align*}
\beta\mapsto (\beta\star\psi)\circ f
\end{align*}
is a bounded operator from $Z(Y)$ to $Z(X)$. Therefore the Banach space $Z(X)$ of cocycles (up to renorming) is a QI invariant. 

It follows that for every kernel $\psi$ and $\rho\ge 0$, there exists a kernel $\psi'$ and $\rho'\ge 0$ such that for all subsets $E$ and $F$ of $X$,
\begin{align} \label{capQI}
\mathrm{Cap}_{\psi}(E+\rho,F+\rho)\le C\,\mathrm{Cap}_{\psi'}(f(E)+\rho',f(F)+\rho').
\end{align}

\begin{proposition} \label{coarseQI}
Coarse $p$-parabolicity is a QI invariant of $p$-PI spaces.
\end{proposition}

\begin{proof}
Indeed, a set $E$ or $E+\rho$ is bounded (resp. the complement of a bounded set) if and only if $f(E)$, or $f(E)+\rho'$, is. Then Estimate (\ref{capQI}) transfers vanishing of capacities.
\end{proof}

\subsection{QI-invariance of coarse liminal $p$-parabolicity}

Strict $p$-para\-bolicity involves connected sets, a notion which is not directly QI-invariant. Therefore some preliminaries are needed.

\begin{notation} \label{tilde+}
Let $X$ be a metric space and $\rho>0$. For $x\in X$, let $B(x,\rho)$ denote the open ball and let
$$
\tilde B(x,\rho)~ \text{be the connected component of $x$ in }B(x,\rho).
$$
For $E\subset X$, let $E+\rho:=\bigcup_{x\in E}B(x,\rho)
$ denote the tubular $\rho$-neighborhood and
\begin{align*}
E\tilde + \rho:=\bigcup_{x\in E}\tilde B(x,\rho)
\end{align*}
the \emph{connected $\rho$-neighborhood} of $E$.
\end{notation}

\begin{definition} \label{qconn}
Let $\eta:[0,+\infty)\to [0,+\infty)$ be a nondecreasing function. A metric space $X$ is \emph{quasi-connected} with \emph{quasi-connectivity profile} $\eta$ if for every $x\in X$ and $\ell>0$, $B(x,\ell)\subset \tilde B(x,\eta(\ell))$.
\end{definition}

\begin{lemma} \label{E'}
Assume that $X$ and $Y$ are locally connected and quasi-connected. Let $f:X\to Y$ be an $(L,D)$-quasiisometry. Let $\rho\ge D$. Let $E\subset Y$ be a subset such that $E\tilde+\rho$ is connected. Then, for large enough $\rho'$, $f^{-1}(E+D)\tilde+\rho'$ is connected.
\end{lemma}

\begin{proof}

Consider the equivalence relation on $E$ generated by $y\sim y'$ if $\tilde B(y,\rho)\cap \tilde B(y',\rho)\not=\emptyset$. By local connectedness, for each equivalence class $\mathcal{C}$, the set $\mathcal{C}\tilde +\rho$ is open. Distinct classes have disjoint connected $\rho$-neighborhoods $\mathcal{C}\tilde+\rho$. Since $E\tilde+\rho$ is connected, there is only one equivalence class. Therefore every two points of $E$ are joined by a finite sequence $y_j$ of points of $E$ with $d(y_j,y_{j+1})< 2\rho$.

Let $E'=f^{-1}(E+D)$. Then $f(E')\subset E+D$. Conversely, since $f(X)$ is $D$-dense, given $y\in E$, there exists $x\in X$ such that $d(f(x),y)\le D$, so $f(x)\in E+D$, $x\in E'$. So $E\subset f(E')+D$. 

Given $x,x'\in E'$, let $y,y'\in E$ be $D$-close to $f(x),f(x')$, let $y_j$ be a $2\rho$ chain in $E$ joining them, let $x_j\in E'$ be such that $d(f(x_j),y_j)\le D$. Then 
\begin{align*}
d(x_{j},x_{j+1})\le  L(D+d(f(x_{j}),f(x_{j+1})))\le L(3D+2\rho).
\end{align*}
Let $\ell=L(3D+2\rho)$ and $\rho'=\eta(\ell)$. For every $j$, $x_{j+1}\in B(x_i,\ell)\subset \tilde B(x_j,\rho')$, so $x,x'$ belong to the same connected component of $E'\tilde+\rho'$. We conclude that $E'\tilde+\rho'$ is connected.

\end{proof}

\begin{definition}
Let $\mathcal{S}_{\rho}$ denote the collection of subsets $E$ of $X$ such that the connected $\rho$-tubular neighborhood $E\tilde+\rho$ is escaping. Say $X$ is \emph{coarsely liminal $p$-parabolic} if $X$ is coarsely $p$-parabolic, and there exists a kernel $\psi$ such that for all $\rho>0$, there exists $\rho'>0$ such that $\mathrm{Cap}_\psi(E+\rho',F+\rho')=+\infty$ for all $E,F\in \mathcal{S}_{\rho}$.
\end{definition}

\begin{proposition} \label{coarsestrictQI}
Coarse liminal $p$-parabolicity is a QI invariant of locally connected, quasi-connected $p$-PI spaces.
\end{proposition}

\begin{proof}

Let $f:X\to Y$ be an $(L,D)$-quasiisometry. 
Let $E,F\in \mathcal{S}_{\rho}(Y)$. Let $E'=f^{-1}(E+D)$, $F'=f^{-1}(F+D)$. According to Lemma \ref{E'}, $E',F'\in \mathcal{S}_{\rho_1}(X)$ for some $\rho_1$ depending only on $L,D,\rho$ and the quasi-connectivity profile $\eta$. 
Assume that $X$ is coarsely liminal $p$-parabolic. There is a $\rho_2\ge \rho_1$ such that $\mathrm{Cap}_{\psi}(E'+\rho_2,F'+\rho_2)=+\infty$. 
Estimate (\ref{capQI}) implies that there exists a suitable kernel $\psi'$ and $\rho_3>0$ such that
\begin{align*}
\mathrm{Cap}_{\psi}(E'+\rho_2,F'+\rho_2)
&\le C\,\mathrm{Cap}_{\psi'}(f(E')+\rho_3,f(F')+\rho_3)\\
&\le C\,\mathrm{Cap}_{\psi'}(E+D+\rho_3,F+D+\rho_3).
\end{align*}
Then, for $\rho'=D+\rho_3$,
$$
\mathrm{Cap}_{\psi'}(E+\rho',F+\rho')=+\infty.
$$
This shows that $Y$ is coarsely liminal $p$-parabolic. 

\end{proof}

\subsection{Comparing capacities and coarse capacities}

Kernels are also used to regularize locally integrable functions
\begin{align*}
(u\star\psi)(x):=\int_{X}u(z)\psi(x,z)\,dz,
\end{align*}
and cocycles,
\begin{align*}
(\alpha\star\psi)(x,y):=\int_{X\times X}\alpha(z,w)\psi(x,z)\psi(y,w)\,dz\,dw.
\end{align*}

\begin{lemma} \label{cap<coarse}
Let $X$ be a $p$-PI space. For every $\rho>0$, there exists a kernel $\psi$ and a constant $C$ such that, for all subsets $E,F\subset X$,
\begin{align*}
\mathrm{Cap}_{p}(E,F)\le C\,\mathrm{Cap}_{\psi}(E+\rho,F+\rho).
\end{align*}
\end{lemma}

\begin{proof}
Start with a kernel $\psi_0$ of width $\rho$. Let $u$ be a locally integrable function on $X$, let $\alpha=(du)\star\psi$, where $du$ denotes the cocycle $du(x,y):=u(x)-u(y)$. Since $\psi_0$ is Lipschitz, so is $u\star\psi_0$, its local Lipschitz constant satisfies 
\begin{align*}
g(x):=Lip_x(u\star\psi)&\le\int_{B(x,\rho)}|u(z)-u(x)|(Lip_x\psi_0)(x,z)\,dz,
\end{align*}
where $\rho$ is the width of the support of $\psi_0$. $g$ is an upper gradient for $u\star\psi_0$ and satisfies, by H\"older's inequality,
\begin{align*}
\left(\int_{X}g^p\right)^{1/p}\le C\,\left(\int_{\{d(x,z)\le\rho\}}|u(x)-u(z)|^p \,dx\,dz \right)^{1/p}\le C\,N^{\psi}(du),
\end{align*}
for another kernel $\psi$ (Lemma 1.10 in \cite{Pansu_Torino}). 

Assume in addition that $u\ge 1$ on $E+\rho$ and $u\le 0$ on $F+\rho$. Then $u\star\psi_0\ge 1$ on $E$ and $u\star\psi_0\le 0$ on $F$. This shows that
\begin{align*}
\mathrm{Cap}_{p}(E,F)\le C\,\mathrm{Cap}_{\psi}(E+\rho,F+\rho).
\end{align*}

\end{proof}

\begin{lemma} \label{coarse<cap}
Let $X$ be a $p$-PI space. There exists a kernel $\psi$ and a constant $C$ such that, for every subset $E\subset X$,
\begin{align*}
\mathrm{Cap}_\psi(E)\le C\,\mathrm{Cap}_p(E).
\end{align*}
\end{lemma}

\begin{proof}
Let $r=R/\sigma$. Let $u$ be a locally integrable boundedly supported competitor for the left hand side. Let $g$ be an upper gradient for $u$. Poincar\'e's $(p,p)$-inequality
\begin{align*}
\left(\int_{B(x,r)} |u-u_B|^p \right)^{1/p} 
&\le C_{p{\mathrm P}}(r) \left( \int_{B(x,\sigma r)} g^p \right)^{1/p}
\end{align*}
shows that, for a suitable constant $C$,
\begin{align*}
\int_{d(x,y)\le r}|u(x)-u(z)|^p
&=\int_{X}\left(\int_{B(x,r)}|u(x)-u(z)|^p\,dz\right)\,dx\\
&\le \int_{X}\left(\int_{B(x,r)}g(z)^p\,dz\right)\,dx\\
&=\int_{X}g(z)^p\mu(B(z,r))\,dz
\le C\,\int g^p.
\end{align*}
The left hand side can be replaced with $c\,N^\psi(du)$ for some kernel $\psi$ and $c>0$.
This shows that
\begin{align*}
\mathrm{Cap}_\psi(E)\le C\,\mathrm{Cap}_p(E).
\end{align*}

\end{proof}

\subsection{QI invariance of $p$-parabolicity}

When $p=2$, $2$-parabolicity is equivalent to vanishing of reduced $L^2$ $1$-cohomology. In the case of Riemannian manifolds, the QI invariance of this property is due to M. Kanai, \cite{Kanai}.

In \cite{Holo}, I. Holopainen proves the QI invariance of Property $D_p$, i.e. vanishing of the space of $p$-harmonic functions with finite $p$-Dirichlet integral. Equivalently, of the reduced $L^p$ $1$-cohomology. When $p\not=2$, it is unclear whether this is equivalent to $p$-parabolicity (an exact $L^p$-cohomological characterization of $p$-parabolicity appears in \cite{GoldTroyDuality}, but its QI invariance does not follow from existing $L^p$-cohomological results). Therefore we provide a direct proof, valid for all $p\ge 1$.

\begin{proposition}
\label{QIinv1}
Let $p\ge 1$. Among $p$-PI spaces, $p$-parabolicity is equivalent to coarse $p$-parabolicity, hence it is a QI invariant property by Proposition \ref{coarseQI}.
\end{proposition}

\begin{pf}

Assume that $X$ is coarsely $p$-parabolic.
According to Lemma \ref{cap<coarse}, for a suitable kernel $\psi$ and radius $\rho$, for every subsets $E,F\subset X$,
\begin{align*}
\mathrm{Cap}_{p}(E,F)\le C\,\mathrm{Cap}_{\psi}(E+\rho,F+\rho).
\end{align*}
If $X\setminus F$ is bounded, so is $X\setminus (F+\rho)$, so
\begin{align*}
\mathrm{Cap}_{p}(E)\le C\,\mathrm{Cap}_{\psi}(E+\rho).
\end{align*}
If $E$ is bounded, so is $E+\rho$, so both capacities vanish. Thus $X$ is $p$-parabolic.

Conversely, assume that $X$ is $p$-parabolic. According to Lemma \ref{coarse<cap}, there exists a kernel $\psi$ such that for every subset $E\subset X$,
\begin{align*}
\mathrm{Cap}_\psi(E)\le C\,\mathrm{Cap}_p(E).
\end{align*}
If $E$ is bounded, $\mathrm{Cap}_p(E)=0$ so $\mathrm{Cap}_\psi(E)=0$. One concludes that $X$ is coarsely $p$-parabolic.

\end{pf}

\subsection{Thickening condensers}

Classically (see for instance \cite{Ferrand_GeoDed} for Riemannian manifolds, and references therein), capacities can be estimated from below as follows: first show that in the definition, $L^1_{loc}$ competitors can be replaced with monotone continuous functions. Then prove an apriori modulus of continuity for monotone continuous functions with $L^Q$ upper gradient. This leads to the following thickening procedure: condensers with very low capacity can be thickened without substantially changing the capacity. This is possible in general ul$Q$bg spaces, for exponents $p\ge Q$. The result is possibly non sharp, since, for $n$-dimensional Riemannian manifolds, the estimate holds as soon as $p>n-1$.

\begin{definition}
Let $X$ be a metric space. Say that $X$ is \emph{quasi-convex} if there exists a constant $C_{qc}$ such that all pairs of points at distance $\ell$ are joined by curves of length $\le C_{qc}\ell$.
\end{definition}
Quasi-convexity implies quasi-connectivity with linear profile.

\begin{proposition} \label{thick}
There exist constants $C$ and $\kappa$ depending only on the ul$Q$bg data, on the quasi-convexity rate $C_{qc}$, on $p\ge Q>1$ and on $r>0$ such that, for every quasi-convex and ul$Q$bg space $X$ and every closed connected subsets $E,F\subset X$ of diameters $\ge 2r$ such that $\mathrm{Cap}_p(E,F)<\kappa$,
\begin{align*}
\mathrm{Cap}_p(E,F)\le \mathrm{Cap}_p(E+r,F+r)\le C\,\mathrm{Cap}_p(E,F).
\end{align*}
\end{proposition}


\begin{proof}

According to Corollary 3.11 in \cite{FLDNGOP}, there exists $r>0$, depending on the ul$Q$bg data only, such that monotone continuous functions satisfy a modulus of continuity of the form
\begin{align*}
\forall x\in X,\quad osc_{B(x,r)}u \le C\,\left(\int_{B(x,Cr)}g^Q\right)^{1/Q}.
\end{align*}
If $p\ge Q$, H\"older's inequality yields
\begin{align*}
\forall x\in X,\quad osc_{B(x,r)}u \le C\,r^{(Q-p)/p}\left(\int_{B(x,Cr)}g^p\right)^{1/p},
\end{align*}
where the power of $r$ arises from the measure of $r$-balls.

Lebesgue's straightening Lemma (Proposition 3.14  in \cite{FLDNGOP}) allows to use only continuous functions which are monotone on $X\setminus(E\cup F)$ as competitors in the definition of capacity. Let $x\in E$. Assuming that the diameters of the connected closed sets $E$ and $F$ are $>2r$, a continuous function which is constant on $E$ and $F$ and monotone on $B(x,r)\setminus(E\cup F)$ is monotone on $B(x,r)$ (Lemma 3.4 in \cite{FLDNGOP}).

Let $u$ be a monotone continuous function such that $u(E)\ge 1$ and $u(F)\le 0$. If the right hand side is $<\frac{1}{3}$, then $v=3u-1$ satisfies $v\ge 1$ (resp. $v\le 0$) on all $r$-balls centered on $E$ (resp. on $F$), hence on $E+r$ (resp. on $F+r$). This shows that if $\mathrm{Cap}_p(E+r,F+r)$ is small enough, then
\begin{align*}
\mathrm{Cap}_p(E+r,F+r)\le C\,\mathrm{Cap}_p(E,F),
\end{align*}
for a constant $C$ depending on ul$Q$bg data only.  When $X$ is quasiconvex, iterating the estimate allows to extend it to arbitrary large values of $r$.
\end{proof}

\subsection{QI invariance of liminal $p$-parabolicity}

\begin{lemma} \label{small}
Let $X$ be a strictly $p$-parabolic space. Let $r>0$. Then there exist pairs of escaping subsets $E,F\subset X$ with arbitrarily small $p$-capacity.
\end{lemma}

\begin{proof}

By assumption, there exist escaping subsets $E,F\subset X$ such that $Cap_p(E,F)<+\infty$. Fix $\epsilon>0$. Let $u$ be a competitor for $Cap_p(E,F)$ with upper gradient $g$ such that $\int g^p<+\infty$. We can assume that $0\le u \le 1$. Let $E'$ be a bounded set such that
\begin{align*}
\int_{X\setminus E'}g^p<\epsilon.
\end{align*}
Since $X$ is $p$-parabolic, there exists a set $F'$ with bounded complement such that $\mathrm{Cap}_p(E',F')<\epsilon$. Let $u'$ be a competitor for this condenser, with upper gradient $g'$ such that $\int g'^p<\epsilon$. One can assume that $0\le u' \le 1$. Let $v=u(1-u')$. Then $g+g'$ is an upper gradient for $v$. Since $v$ vanishes on $E'$, 
\begin{align*}
\int_X (g+g')^p\le 2^p(\int_{X\setminus E'}g^p+\int_X g'^p)<2^{p+1}\eps.
\end{align*}
Since $v$ is a competitor for the condenser $E\cap F',F\cap F'$, we conclude that this condenser has small capacity. By definition of escaping, $E\cap F',F\cap F'$ contain escaping subsets, of even lower capacity.

\end{proof}

\begin{proposition}
\label{QIinv2}
Let $p\ge Q> 1$. Among locally connected, quasi-convex ul$Q$bg and $p$-PI spaces, liminal $p$-parabolicity is equivalent to coarse liminal $p$-parabolicity, hence it is a QI-invariant property by Proposition \ref{coarsestrictQI}.
\end{proposition}

\begin{proof}

Let $X$ be a $p$-parabolic space. According to Proposition \ref{QIinv1}, $X$ is also coarsely $p$-parabolic.

First, assume that $X$ is liminal $p$-parabolic. According to Lemma \ref{cap<coarse}, for a suitable kernel $\psi$ and radius $r$, for every subsets $E_0,F_0\subset X$,
\begin{align*}
\mathrm{Cap}_{p}(E_0,F_0)\le C\,\mathrm{Cap}_{\psi}(E_0+r,F_0+r).
\end{align*}
Let $\rho>0$. If $E+\rho$ and $F+\rho$ are escaping, 
\begin{align*}
+\infty=\mathrm{Cap}_{p}(E+\rho,F+\rho)\le C\,\mathrm{Cap}_{\psi}(E+\rho+r,F+\rho+r).
\end{align*}
Setting $\rho'=\rho+r$, we see that $X$ is coarsely liminal $p$-parabolic.

Second, assume that $X$ is not liminal $p$-parabolic, i.e. $X$ is strictly $p$-parabolic. Choose $r=0$ and fix $r'\ge 0$. Lemma \ref{small} provides us with escaping subsets $E,F\in\mathcal{S}_0(X)$ such that $\mathrm{Cap}_p(E,F)<\kappa$, the bound provided by Proposition \ref{thick}. This Proposition ensures that 
$$
\mathrm{Cap}_p(E+r',F+r')\le C\,\mathrm{Cap}_p(E,F)<+\infty.
$$
Since $X$ is $p$-PI, Lemma \ref{coarse<cap} provides a kernel $\psi$ such that
$$
\mathrm{Cap}_\psi(E+r',F+r')\le C\,\mathrm{Cap}_p(E+r',F+r')<+\infty.
$$
Thus $X$ is coarsely strictly $p$-parabolic, hence not coarsely liminal $p$-parabolic. 

This proves the equivalence of liminal $p$-parabolicity and coarse liminal $p$-parabolicity.
\end{proof}

\section{Rumin conformal cohomology}
\label{rumin}

The goal of this section is to describe Rumin conformal cohomology, to relate it to QI invariants and to prove the following QC invariance statement.

\begin{proposition} \label{qcRumin}
The Rumin conformal cohomology is a quasiconformal invariant of contact subRiemannian manifolds.
\end{proposition}

\subsection{Rumin $L^{q,p}$ cohomology}

Let $M$ be a smooth $2n+1$-dimensional manifold equipped with a contact structure $\xi$, viewed as a subbundle of the tangent bundle. In the differential graded algebra of smooth differential forms $\Omega^\cdot$, let $\mathcal{I}^\cdot$ be the graded differential ideal generated by $1$-forms that vanish on $\xi$. Then $\mathcal{I}^k=\Omega^k$ in degrees $k>n$.
Let $\mathcal{J}$ be its annihilator, i.e. smooth forms whose wedge products with elements of $\mathcal{I}^\cdot$ of complementary degree vanish. Then $\mathcal{J}^k$ vanishes in degrees $k\le n$. The exterior differential induces differentials $d_c$ on $\Omega^\cdot/\mathcal{I}^\cdot$ and on $\mathcal{J}^\cdot$. In \cite{Rumin}, M. Rumin introduced a second order differential operator, still denoted by $d_c:\Omega^n/\mathcal{I}^n\to\mathcal{J}^{n+1}$, which completes the sequence 
\begin{align*}
0\to\Omega^0/\mathcal{I}^0\stackrel{d_c}{\to}\cdots\stackrel{d_c}{\to}\Omega^n/\mathcal{I}^n\stackrel{d_c}{\to}\mathcal{J}^{n+1}\stackrel{d_c}{\to}\cdots\stackrel{d_c}{\to}\mathcal{J}^{2n+1}\to 0
\end{align*}
into a cochain complex called \emph{Rumin's complex}. This new cochain complex is chain homotopic to the de Rham complex. M. Rumin shows that the whole complex is invariant under smooth contact transformations.

Let us introduce metrics. A \emph{subRiemannian metric $g$} on $M$ is a smooth metric on the fibers of the subbundle $\xi$. Locally, $\Omega^\cdot/\mathcal{I}^\cdot$ and $\mathcal{J}^\cdot$ can be viewed as the spaces of smooth sections of a graded bundle $E^{\cdot}$. The subRiemannian metric determines a smooth metric on the fibers of $E^\cdot$ and a volume form. Therefore one can speak of $L^p$ sections of $E^\cdot$, called \emph{$L^p$ Rumin forms}. Given a sequence $\mathbf{p}=(p_k)\in [1,\infty]$, the differential operators $d_c$ extend to 
\begin{align*}
L^{\mathbf{p}}E^{\cdot}:=\bigoplus_{k}\{\omega\in L^{p_k} E^k\,;\,d_c\omega\in L^{p_{k+1}} E^{k+1}\},
\end{align*}
where the statement $d_c\omega\in L^{p_{k+1}} E^{k+1}$ is taken in distributional sense. This is still a complex, whose cohomology, denoted by
\begin{align*}
L^{\mathbf{p}}H^{\cdot}(M,\xi,g)
\end{align*} 
is called the $L^{\mathbf{p}}$ cohomology of the subRiemannian contact manifold $(M,\xi,g)$.

\subsection{QI invariance of Rumin $L^{q,p}$ cohomology}

In \cite{BFP}, it is shown that, under conditions on the sequence $\mathbf{p}$, the $L^{\mathbf{p}}$ Rumin complex is locally exact. In \cite{Pansu_cup}, under additional bounded geometry assumptions, this local information is turned into a global chain homotopy of the $L^{q,p}$-Rumin complex with a QI invariant chain complex, the $\ell^{q,p}$ complex of simplicial cochains of a polyhedral approximation of the metric space $(M,g)$. This last complex will not be defined in detail, and called the $\ell^{q,p}$-complex of $M$ for short.

Here, to avoid sophistications, we merely state nonlimiting cases.

\begin{proposition}[Theorem 4 in \cite{Pansu_cup}] \label{qir}
Let $1\le h\le 2n+1$. Let $\mathbf{p}$ be a nonincreasing sequence in $(1,\infty)$. Assume that $\frac{1}{p_{k}}-\frac{1}{p_{k-1}}\leq \frac{1}{2n+2}$ (to be replaced with $\frac{1}{p_{n+1}}-\frac{1}{p_{n}}\leq \frac{2}{2n+2}$ in degree $k=n+1$). 

Consider the class of $2n+1$ dimensional contact subRiemannian manifolds with the following properties:
\begin{enumerate}
  \item Bounded geometry.
  \item Uniform vanishing of cohomology of to degree $h$.
\end{enumerate}

For $(M,\xi,g)$ in this class, the $L^{q,p}$-Rumin complex and the $\ell^{q,p}$-complex of $M$ are chain homotopic up to degree $h$. In particular, the cohomology spaces  
\begin{align*}
L^{\mathbf{p}}H^{\cdot}(M,\xi,g)\quad\text{and}\quad E\ell^{\mathbf{p}}H^{\cdot}(M)
\end{align*}
(the subspace of $\ell^{\mathbf{p}}$ cohomology represented by exact cocycles) are isomorphic up to degree $h+1$.
\end{proposition}

Combining Proposition \ref{qir} with \cite{Ducret} leads to

\begin{corollary} \label{QIHk}
Let $1\le h\le 2n+1$. Let $1<p\le q<\infty$, $\frac{1}{p}-\frac{1}{q}\leq \frac{1}{2n+2}$ (to be replaced with $\frac{1}{p}-\frac{1}{q}\leq \frac{2}{2n+2}$ if $k=n+1$). In the above class of subRiemannian contact manifolds, for $k\le h+1$, $L^{q,p}H^{k}(M,\xi,g)$ is a QI invariant.

\end{corollary}

We note that a cocompact conctactomorphic and isometric action of some group on $(M,\xi,g)$ automatically implies the bounded geometry assumption and upgrades vanishing of cohomology into uniform vanishing of cohomology. Here, we shall merely need degree $2$, so only the vanishing of $H^1(M,\R)$ is required. The examples constructed in Section \ref{construction} indeed have vanishing ordinary $H^1$.

\subsection{Rumin conformal cohomology}

\begin{definition} \label{defRcoh}
The \emph{Rumin conformal cohomology} of a $2n+1$-dimensional subRiemannian contact manifold $(M,\xi,g)$ is its Rumin $L^{\mathbf{p}}$ cohomology for the \emph{conformal sequence} defined by
\begin{align*}
p_k=\begin{cases}
\frac{2n+2}{k} & \text{if }k\le n, \\
\frac{2n+2}{k+1}  & \text{otherwise}.
\end{cases}
\end{align*}
\end{definition}

According to Corollary \ref{QIHk}, the Rumin conformal cohomology, except in degrees $0$ and $2n+1$, is a QI invariant of $2n+1$-dimensional subRiemannian contact manifolds with bounded geometry and uniform vanishing of cohomology in suitable degrees.

\subsection{QC invariance of Rumin conformal cohomology}

The Rumin conformal cohomology is invariant under smooth QC contact transformations by design. Indeed, the $L^{p_k}$-norm of Rumin $k$-forms is exactly invariant under conformal changes of subRiemannian metric, and invariant up to bounded multiplicative factors under quasiconformal changes.

The issue is to handle weakly regular homeomorphisms between subRiemannian contact manifolds. It turns out that QC maps have just enough regularity for our purposes.

\subsubsection{First step} QC maps between quasiconvex ul$Q$bg spaces are locally quasisymmetric.

In the plane, this is Gr\"otzsch's theorem. The generalization is due to J. Heinonen and P. Koskela, (Theorem 4.7 in \cite{HK}).

\subsubsection{Second step} Quasisymmetric maps between equiregular subRiemannian manifolds are a.e. differentiable. 

This is an avatar of Rademacher and Stepanov's a.e. differentiability theorem, which can be found in \cite{Pansu_Annals} and \cite{MM}.

\subsubsection{Third step} The Jacobian of a QC map belongs to $L^r$ for some $r>1$. It follows that the first horizontal partial derivatives of a QC map $f$ belong to $L^p$ for a $p>2n+2$, i.e. $f\in W^{1,p}$.

This famous result of F. Gehring in Euclidean spaces has been generalized by J. Heinonen and P. Koskela, (Theorem 7.11 in \cite{HK}).

\subsubsection{Fourth step} $W^{1,p}$ maps between subRiemannian contact manifolds, $p > 2n+2$, induce chain maps on Rumin's complex.

This is Theorem 1.5 in \cite{KMX}. On open subsets of Heisenberg group $\mathbb{H}_n$, the chain map $f_P^*$, called \emph{Pansu pullback}, is defined as follows. Let $\mathfrak{h}=\mathfrak{h}_1 \oplus \mathfrak{h}_2$ be the gradation of the Lie algebra of $\mathbb{H}_n$. Let $\theta$ denote a $1$-form dual to $\mathfrak{h}_2$. Let $L:\bigwedge^{\cdot}\mathfrak{h}_1^*\to \bigwedge^{\cdot}\mathfrak{h}_1^*$ denote wedge multiplication with the symplectic 2-form $\Omega$ defined by the Lie bracket, $\Omega=\theta\circ[\cdot,\cdot]$.
The bundle $E^{\cdot}$ is left-invariant, its fibre is 
\begin{align*}
E^{\cdot\le n}&=\bigwedge^{\cdot}\mathfrak{h}_1^*/(\mathrm{Im}(L)),\\
E^{\cdot> n}&=\theta\wedge(\mathrm{Ker}(L)).
\end{align*}
Graded homomorphisms $\Phi:\mathfrak{h}\to\mathfrak{h}$ act on $\bigwedge^\cdot\mathfrak{h}$ by pull-back. They preserve the subspaces $\R\theta$ and $\R\Omega$, hence $\mathrm{Im}(L)$ and $\mathrm{Ker}(L)$. Thus they act on $E^\cdot$.
Therefore, given a smooth Rumin form $\omega$, if a map $f$ is a.e. differentiable, the Rumin form $x\mapsto Df(x)^*\omega(f(x))$ is a.e. defined, this is $f_P^*\omega$.

\subsubsection{Fifth step} For a QC map, the chain map induced on Rumin forms extends to a bounded operator on $L^{\mathbf{p}}E^{\cdot}$, provided $\mathbf{p}$ is the conformal sequence.

If $f$ is a $K$-QC map, then the graded automorphism $Df$ is $K$-quasiconformal as well. Let $A=D_\mathbb{H}f$ denote its restriction to $\mathfrak{h}_1$, let $\sigma_1\le\cdots\sigma_{2n}$ be its singular values (the square roots of the eigenvalues of $A^\top A$). Then
\begin{align*}
\frac{\sigma_{2n}}{\sigma_{1}}\le K.
\end{align*}
Since $|A|=\sigma_{2n}$ and 
\begin{align*}
\mathrm{det}(A)=\sqrt{\mathrm{det}(A^{\top}A)}=\prod \sigma_i\ge \sigma_1^{2n},
\end{align*}
this implies that
\begin{align*}
|A|^{2n}=\sigma_{2n}^{2n}\le K^{2n}\,\mathrm{det}(A).
\end{align*}
On the other hand, if $(Df)^*\theta=\lambda\theta$, then $(Df)^*\Omega=\lambda\Omega$, hence
\begin{align*}
\mathrm{Jac}(f)\theta\wedge\Omega^n
=(Df)^*\theta\wedge ((Df)^*\Omega)^n
=\lambda^{1+n}\theta\wedge\Omega^n.
\end{align*}
Also,
\begin{align*}
\mathrm{det}(A)\Omega^n
=(D_{\mathbb{H}}f^*\Omega)^n
=\lambda^n\Omega^n,
\end{align*}
so
\begin{align*}
\mathrm{Jac}(f)=\mathrm{det}(A)^{(n+1)/n},
\end{align*}
thus
\begin{align*}
|D_{\mathbb{H}}f|^{2n+2}\le K^{2n+2}\mathrm{Jac}(f).
\end{align*}
As an operator on $E^{k}$, $Df$ has norm
\begin{align*}
|E^{k}Df|&\le|\bigwedge^k D_{\mathbb{H}}f| \le K^k\mathrm{Jac}(f)^{k/(2n+2)} \quad \text{if }k\le n,\\
|E^{k}Df|&\le\lambda|\bigwedge^{k-1} D_{\mathbb{H}}f| \le K^{k-1}\mathrm{Jac}(f)^{(k+1)/(2n+2)} \quad \text{if }k> n.
\end{align*}
For a smooth Rumin $k$-form $\omega$, at a.e. $x$,
\begin{align*}
|f_P^*\omega(x)|^{p_k}\le K^{2n+2}|\omega(f(x))|^{p_k}\mathrm{Jac}(f)(x).
\end{align*}
Since QC maps are absolutely continuous (\cite{MM}), the change of variable formula holds and yields
\begin{align*}
\int |f_P^*\omega|^{p_k}\le K^{2n+2}\int |\omega|^{p_k}
\end{align*}
This shows that Pansu pullback preserves $L^{p_k}$ norms up to a factor depending on  $K$ only. Since, according to \cite{KMX}, pullback commutes with $d_c$, it extends to a bounded operator on $L^{\mathbf{p}}E^{\cdot}$. 

\subsubsection{Conclusion} QC equivalent subRiemannian contact manifolds have homotopic conformal Rumin complexes, hence isomorphic conformal Rumin cohomologies. 

Indeed, the inverse map of a QC homeomorphism is itself QC, since it is quasisymmetric. This completes the proof of Proposition \ref{qcRumin}.

\section{Lie algebra cohomology}
\label{lie}

This is the cohomology of the exterior differential restricted to left-invariant differential forms. This is purely algebraic. For instance, if $\omega$ is a left-invariant $1$-form and $X,Y$ are left-invariant vector fields, 
\begin{align}\label{d}
d\omega(X,Y)=-\omega([X,Y]).
\end{align}

\subsection{Weights}

Say a left-invariant differential form $\omega$ has weight $w$ if $\delta_t^*\omega=e^{tw}\omega$. Forms in each degree split as sums of homogeneous forms,
$$
{\bigwedge}^{k}\mathfrak{g}^*=\bigoplus_{w}{\bigwedge}^{k,w}.
$$
Weight is additive under wedge products, it takes nonzero integer values on $1$-forms, so weights on $k$-forms are integers $\ge k$. For instance, 
$$
{\bigwedge}^{1,w}=\{\omega\in\mathfrak{g}^*\,;\,\omega(X)=0\,\forall X\in \mathfrak{g}_1\oplus\cdots \oplus\mathfrak{g}_{w-1}\oplus \mathfrak{g}_{w+1}\oplus\cdots\oplus\mathfrak{g}_s\}.
$$

Weight is preserved by the exterior differential, hence the Lie algebra cohomology splits as well,
$$
H^{k}(\mathfrak{g})=\bigoplus_{w}H^{k,w}.
$$

\begin{definition}
Let $\mathfrak{g}$ be a Carnot Lie algebra. $w_{min}(\mathfrak{g})$ (resp. $w_{max}(\mathfrak{g})$) is the least (resp. largest) weight $w$ such that $H^{2,w}(\mathfrak{g})$ does not vanish.
\end{definition}

For instance, if $\mathfrak{a}$ is an abelian Lie algebra of dimension $\ge 2$, $w_{min}(\mathfrak{a})=w_{max}(\mathfrak{a})=2$. If $\mathfrak{h}(n)$ denotes the $n$-th Heisenberg Lie algebra, then $w_{min}(\mathfrak{h}(n))=w_{max}(\mathfrak{h}(n))=2$ if $n\ge 2$, but $w_{min}(\mathfrak{h}(1))=w_{max}(\mathfrak{h}(1))=3$. In most cases, $w_{min}(\mathfrak{g})<w_{max}(\mathfrak{g})$.

\subsection{Algebraic criterion for high $w_{min}$}

\begin{notation}
Let $\mathfrak{g}$ be a Lie algebra. The descending central series is defined inductively by $\mathfrak{g}^{(1)}=\mathfrak{g}$ and $\mathfrak{g}^{(i+1)}=[\mathfrak{g},\mathfrak{g}^{(i)}]$.
\end{notation}
By definition, for a Carnot Lie algebra $\mathfrak{g}=\mathfrak{g}_1\oplus\cdots\oplus\mathfrak{g}_s$, 
\begin{align*}
\mathfrak{g}^{(i)}=\mathfrak{g}_{i}\oplus\cdots\oplus\mathfrak{g}_s.
\end{align*}

\begin{lemma}\label{lem: algebraic criterion for high w_min}
A Carnot Lie algebra $\mathfrak{f}$ has $w_{min}(\mathfrak{f})\ge 3$ if and only if its $2$-step quotient $\mathfrak{f}/\mathfrak{f}^{(3)}$ is free.
\end{lemma}

\begin{pf}
$w_{min}(\mathfrak{f})\ge 3$ means that the weight $2$ component $H^{2,2}$ of $H^2(\mathfrak{f})$ vanishes. Weight $2$ $2$-forms are automatically closed (since weight $2$ $3$-forms vanish). So
$$
H^{2,2}={\bigwedge}^{2,2}/(d{\bigwedge}^{1,2}).
$$
Formula (\ref{d}) shows that 
\begin{align*}
H^{2,2}=0 &\iff [\cdot,\cdot]^*: \mathfrak{f}_{2}^*\to{\bigwedge}^2\mathfrak{f}_{1}^*\to  \text{ is onto}\\
&\iff [\cdot,\cdot]: {\bigwedge}^2\mathfrak{f}_{1}\to \mathfrak{f}_{2} \text{ is injective}\\
&\iff \mathfrak{f}_{1}\oplus \mathfrak{f}_{2} \text{ is a free $2$-step Lie algebra}.
\end{align*}

\end{pf}

\begin{remark}
Step $2$ Carnot Lie algebras cannot provide examples.
\end{remark}

Let $\mathfrak{f}$ be a Carnot Lie algebra such that $w_{min}(\mathfrak{f})\ge 3$, and $\mathfrak{g}$ a $2$-step Carnot Lie algebra that has the same dimension $m$ and homogeneous dimension $Q$ as $\mathfrak{f}$. Let $n_1$ and $n_2$ be the dimensions of its levels. The equations $n_1+n_2=m$ and $n_1+2n_2=Q$ yield $n_1=2m-Q$, $n_2=Q-m$. On the other hand, if $p_i=\mathrm{dim}(\mathfrak{f}_i)$, 
\begin{align*}
n_1&=2m-Q=\sum (2-i)p_i\le p_1-p_3,\\
n_2&=Q-m=\sum (i-1)p_i\ge p_2+2p_3.
\end{align*}
Since $\mathfrak{f}/\mathfrak{f}^{(3)}$ is free, $p_2=\frac{p_1(p_1-1)}{2}$, so
\begin{align*}
\frac{p_1(p_1-1)}{2}&\le p_2\le n_2\le \frac{n_1(n_1-1)}{2}\le \frac{(p_1-p_3)(p_1-p_3-1)}{2}\\
&\le \frac{p_1(p_1-1)}{2},
\end{align*}
with equality implying $p_3=0$ and $p_2=n_2$, i.e. $\mathfrak{f}=\mathfrak{g}$ are free $2$-step, $w_{max}(\mathfrak{g})=w_{min}(\mathfrak{f})=3$. Thus a $2$-step Carnot Lie algebra cannot achieve the numerical conditions required by the construction described in Section \ref{construction}.

\subsection{Algebraic criteria for low $w_{max}$}

\begin{remark}
The value $w_{max}(\mathfrak{g})=2$ cannot be relaxed.
\end{remark}

Indeed, in the Riemannian case, $M$ and $N$ have dimension $n$. We need $n$-parabolic manifolds, hence $n\ge Q$. Setting $p=\frac{n}{2}$ and $q=n$, $\frac{1}{p}-\frac{1}{q}=\frac{1}{n}$. In order that $L^{n,n/2}H^2(M)\not= L^{n,n/2}H^2(N)$, one needs that
$$
\frac{w_{max}(\mathfrak{g})-1}{Q}\le \frac{1}{n}<\frac{w_{min}(\mathfrak{f})-1}{Q}.
$$
With the constraint $n\ge Q$, this can happen only if $w_{max}(\mathfrak{g})=2$, with $n=Q$. 

In the contact subRiemannian case, $M$ and $N$ have dimension $2n+1$. We need $2n+2$-parabolic manifolds, hence $2n+2\ge Q$. Setting $p=n+1$ and $q=2n+2$, $\frac{1}{p}-\frac{1}{q}=\frac{1}{2n+2}$. In order that $L^{2n+2,n+1}H^2(M)\not= L^{2n+2,n+1}H^2(N)$, one needs again that $w_{max}(\mathfrak{g})=2$, with $2n+2=Q$.

\subsection{Constructing the two examples}

Since the examples needed for our purposes are two 3-step 16-dimensional Carnot algebras and there is currently no complete classification of nilpotent (or even Carnot) Lie algebras, one needs to make sure that the algebras $\mathfrak{f}$ and $\mathfrak{g}$ we construct are well-defined. Furthermore, while explicit computations of the Lie algebra cohomology are available for complete classifications of nilpotent Lie algebras in low dimensions (see for example \cite{Gong,Magnin1,Magnin2,del Barco,Vergne}), we cannot rely on them in the present setting, and the relevant cohomological computations must therefore be carried out directly.

A first step is to verify that the Jacobi identity is satisfied once the bracket relations are specified. This can be streamlined by studying the action of the Chevalley-Eilenberg differential on left-invariant forms, as explained in the following remark.

\begin{remark}\label{remark: Jacobi with d^2=0}
    The Jacobi identity is equivalent to requiring $d^2=0$ on left-invariant 1-forms.

    For any $\theta\in\mathfrak{g}^\ast$ and any $X,Y,Z\in\mathfrak{g}$, we have
    \begin{align*}
        d^2\theta(X,Y,Z)=&d(d\theta)(X,Y,Z)\\=&-d\theta([X,Y],Z)+d\theta([X,Z],Y)-d\theta([Y,Z],X)\\=&\theta([[X,Y],Z])-\theta([[X,Z],Y])+\theta([[Y,Z],X])\\=&\theta([[X,Y],Z]+[[Y,Z],X]+[[Z,X],Y])=0
    \end{align*}
    Moreover, this means that $d^2$ always vanishes when applied to left-invariant forms in $\mathfrak{g}_2^\ast$, so it is sufficient to check this condition on elements in $\mathfrak{g}_3^\ast$. 
\end{remark}
We are going to present the two Carnot Lie algebras $\mathfrak{f}$ and $\mathfrak{g}$ by providing a choice of a basis together with the list of non-zero bracket relations. Since these two examples are 16-dimensional, we will use a different letter for the elements of the basis depending on the layer they belong to, with the aim of making the construction somewhat clearer to the reader. Namely, we will denote by $X_i, T_j$ and $W_k$ the basis elements in the first, second, and third layer, respectively. Moreover, we will denote their duals by $X_i^\ast=dx_i$, $T_j^\ast=\tau_j$, and $W_k^\ast=\omega_k$. 
\subsubsection{Constructing $G$}
Let us start by presenting the Carnot Lie algebra $\mathfrak{g}$ with $w_{max}=2$ since this will require much more work than $\mathfrak{f}$. We are aiming to construct a 16-dimensional Carnot Lie algebra with stratification $\mathfrak{g}_1\oplus\mathfrak{g}_2\oplus\mathfrak{g}_3$, and homogeneous dimension $Q=28$. This can be achieved by taking
\begin{align*}
    \operatorname{dim}(\mathfrak{g}_1)=8\ ,\ \operatorname{dim}(\mathfrak{g}_2)=4\ ,\ \operatorname{dim}(\mathfrak{g}_3)=4\,.
\end{align*}
This particular choice of dimensions will become apparent after the following considerations.

In the case of a 3-step stratification, the Lie algebra cohomology will appear a priori in several possible weights. In weight 2, we have
\begin{align*}
    H^{2,2}\mathfrak{g}^\ast={\bigwedge}^{2,2}\mathfrak{g}^\ast/\big(d{\bigwedge}^{1,2}\mathfrak{g}^\ast\big)
\end{align*}
since $\ker d\cap{\bigwedge}^{2,2}\mathfrak{g}^\ast={\bigwedge}^{2,2}\mathfrak{g}^\ast$ of dimension $\binom{\operatorname{dim}(\mathfrak{g}_1)}{2}$, while the subspace $d{\bigwedge}^{1,2}\mathfrak{g}^\ast$ has dimension equal to $\operatorname{dim}(\mathfrak{g}_2)$.

In weight 3, we have
\begin{align*}
    H^{2,3}\mathfrak{g}^\ast=\big(\ker d\cap{\bigwedge}^{2,3}\mathfrak{g}^\ast\big)/\big(d{\bigwedge}^{1,3}\mathfrak{g}^\ast\big)\,.
\end{align*}
In this case, we have that $\ker d\cap{\bigwedge}^{2,3}\mathfrak{g}^\ast$ is a subspace of ${\bigwedge}^{2,3}\mathfrak{g}^\ast$ of dimension $\operatorname{dim}(\mathfrak{g}_1)\cdot\operatorname{dim}(\mathfrak{g}_2)$. Similarly to before, the dimension of $d{\bigwedge}^{1,3}\mathfrak{g}^\ast$ is equal to $\operatorname{dim}(\mathfrak{g}_3)$.

For weight greater than or equal to 4, the expression for the Lie algebra cohomology simplifies to
\begin{align*}
    H^{2,4}\mathfrak{g}^\ast=\ker d\cap{\bigwedge}^{2,4}\mathfrak{g}^\ast\,.
\end{align*}
The requirement $w_{max}=2$ then translates to the following two conditions
\begin{itemize}
    \item[i.] $\ker d\cap{\bigwedge}^{2,3}\mathfrak{g}^\ast=d{\bigwedge}^{1,3}\mathfrak{g}^\ast$, and
    \item[ii.] the map $d\colon{\bigwedge}^{2,4}\mathfrak{g}^\ast\longrightarrow{\bigwedge}^{3,4}\mathfrak{g}^\ast$ is injective.
\end{itemize}
In our specific example, since $\operatorname{dim}(\mathfrak{g}_3)=4$, the subspace $\ker d\cap{\bigwedge}^{2,3}\mathfrak{g}^\ast$ needs to be 4-dimensional, while $\ker d\cap{\bigwedge}^{2,4}\mathfrak{g}^\ast$ has to be trivial.

Let us denote by $\{X_1,X_2,X_3,X_4,X_5,X_6,X_7,X_8\}$ a basis of the 8-dimensional first layer $\mathfrak{g}_1$. Since we are requiring $\operatorname{dim}(\mathfrak{g}_2)=4$, let us denote by $\{T_1,T_2,T_3,T_4\}$ its basis. Let us impose
\begin{align*}
    [X_1,X_2]=[X_3,X_4]=T_1\ ,\ [X_2,X_3]=[X_4,X_5]=T_2\\
    [X_3,X_5]=[X_4,X_6]=T_3\ ,\ [X_1,X_3]=[X_2,X_4]=[X_5,X_6]=T_4\,,
\end{align*}
so that
\begin{itemize}
    \item $d\tau_1=-dx_1\wedge dx_2-dx_3\wedge dx_4$;
    \item $d\tau_2=-dx_2\wedge dx_3-dx_4\wedge dx_5$;
    \item $d\tau_3=-dx_3\wedge dx_5-dx_4\wedge dx_6$;
    \item $d\tau_4=-dx_1\wedge dx_3-dx_2\wedge dx_4-dx_5\wedge dx_6$.
\end{itemize}
Let us first study the differential map 
\begin{align*}
    d\colon{\bigwedge}^{2,3}\mathfrak{g}^\ast\longrightarrow{\bigwedge}^{3,3}\mathfrak{g}^\ast
\end{align*}
from the $8\cdot 4=28$-dimensional space of 2-forms of weight 3 to the $\binom{8}{3}=56$-dimensional space of 3-forms of weight 3. Direct computations show that the kernel of this map is a 4-dimensional space, just as we needed, and a convenient basis can be taken as
\begin{align*}
-dx_1\wedge\tau_1+dx_4\wedge\tau_4+dx_5\wedge\tau_3\ ,\ -dx_3\wedge\tau_2-dx_4\wedge\tau_3\\
-dx_2\wedge\tau_1-dx_3\wedge\tau_4+dx_6\wedge\tau_3\ ,\ dx_2\wedge\tau_1-dx_4\wedge\tau_2\,.
\end{align*}
It is therefore sufficient to take the 4-dimensional $\mathfrak{g}_3$ spanned by
\begin{align*}
    W_1=&[X_1,T_1]=-[X_5,T_3]=-[X_4,T_4]\ ,\\
    W_2=[X_3,&T_2]=[X_4,T_3]\ ,\ W_3=[X_3,T_4]=-[X_6,T_3]\ ,\\
    W_4=&[X_4,T_2]=[X_3,T_4]-[X_2,T_1]=W_3-[X_2,T_1]\,,
\end{align*}
so that
\begin{itemize}
    \item $d\omega_1=-dx_1\wedge\tau_1+dx_4\wedge\tau_4+dx_5\wedge\tau_3$;
    \item $d\omega_2=-dx_3\wedge\tau_2-dx_4\wedge\tau_3$;
    \item $d\omega_3=-dx_2\wedge\tau_1-dx_3\wedge\tau_4+dx_6\wedge\tau_3$;
    \item $d\omega_4=dx_2\wedge\tau_1-dx_4\wedge\tau_2$.
\end{itemize}
For these choices of brackets, we get that
\begin{align*}
    d{\bigwedge}^{1,3}\mathfrak{g}^\ast=\ker d\cap{\bigwedge}^{2,3}\mathfrak{g}^\ast\,,
\end{align*}
and so requirement i. is satisfied. Notice that the fact that $d{\bigwedge}^{1,3}\mathfrak{g}^\ast=\ker d\cap{\bigwedge}^{2,3}\mathfrak{g}^\ast$ readily implies that $d^2=0$ on 1-forms of weight 3, and so the Jacobi identity for the Lie algebra $\mathfrak{g}$ is satisfied by Remark \ref{remark: Jacobi with d^2=0}.

Finally, we are left to check that requirement ii. also holds for such a choice of Lie brackets. The space of 3-forms of weight 4 is spanned by forms in ${\bigwedge}^{2}\mathfrak{g}_1^\ast\otimes\mathfrak{g}_2^\ast$. The contributions arising from the generators $dx_i\wedge\omega_j$ and $\tau_k\wedge\tau_l$ are sufficiently independent that no nontrivial linear combination can be annihilated by the differential. Equivalently, the structure constants defining the third layer introduce no additional relations among 2-forms of weight 4 beyond those already encoded by skew-symmetry. 
It then follows that the differential $d\colon{\bigwedge}^{2,4}\mathfrak{g}^\ast\to{\bigwedge}^{3,4}\mathfrak{g}^\ast$ is injective. This injectivity can be verified by an explicit computation, which may be carried out using a computer algebra system such as Sage \cite{GIT}.

Combining the computations in all homogeneous weights, we obtain
$$H^{2,w}\mathfrak g=0
\ \text{ for }\ w=3,4,5,6\,,$$
while
$$H^{2,2}\mathfrak g ={\bigwedge}^{2,2}\mathfrak{g}^\ast/ \big(d{\bigwedge}^{1,2}\mathfrak{g}^\ast\big)={\bigwedge}^2\mathfrak{g}_1/(d\,\mathfrak{g}_2^\ast)$$
so that
$$H^2\mathfrak g=H^{2,2}\mathfrak g$$
with
$\operatorname{dim}( H^{2,2}\mathfrak g)=28-4=24$, and so $w_{max}=2$.

\subsubsection{Constructing $F$}
We are looking for a 3-step nilpotent Lie group $F$ whose Carnot Lie algebra $\mathfrak{f}$ has $w_{min}(\mathfrak{f})\ge 3$. According to Lemma \ref{lem: algebraic criterion for high w_min}, this implies that $\mathfrak{f}/\mathfrak{f}_3$ is free. In other words, this imposes a strict condition on the dimension of the second layer $\mathfrak{f}_2$, namely $$\operatorname{dim}(\mathfrak{f}_2)=\binom{\operatorname{dim}(\mathfrak{f}_1)}{2}\,.$$

For our example, we are specifically looking for a 16-dimensional Carnot Lie algebra $\mathfrak{f}$ with homogeneous dimension $Q=28$. This is achieved by taking
\begin{align*}
    \operatorname{dim}(\mathfrak{f}_1)=5\ ,\ \operatorname{dim}(\mathfrak{f}_2)=\binom{5}{2}=10\ ,\ \operatorname{dim}(\mathfrak{f}_3)=1\,.
\end{align*}
An example of such a Carnot Lie algebra $\mathfrak{f}$ is given by the following non-trivial Lie brackets:
\begin{align*}
    [X_1,X_2]=T_1\ ,\ [X_1,X_3]=T_2\ ,\ [X_1,X_4]=T_3\ ,\ [X_1,X_5]=T_4\\
    [X_2,X_3]=T_5\ ,\ [X_2,X_4]=T_6\ ,\ [X_2,X_5]=T_7 \ ,\ [X_3,X_4]=T_8\\
    [X_3,X_5]=T_9\ ,\ [X_4,X_5]=T_{10}\ ,\ [X_1,T_1]=W_1\quad\quad
\end{align*}
This is a well-defined Lie algebra since 
\begin{align*}
    d^2 \omega_1=-d(dx_1\wedge\tau_1)=-dx_1\wedge dx_1\wedge dx_2=0\,.
\end{align*}
Notice that we could have chosen some alternative bracket relation for the element $W$. For example, the choice $[X_1,T_1]=[X_4,T_8]=\tilde{W}_1$ would have also worked since $d^2\tilde{\omega}_1=-d(dx_1\wedge\tau_1+dx_4\wedge\tau_8)=0$. This means that there are several non-isomorphic Carnot Lie algebras of dimension 16 that would suit our needs (as long as $\operatorname{dim}(\mathfrak{f}_3)=1$).

\end{document}